\documentclass[12pt]{article}

 \usepackage{amsmath,amssymb,amscd,amsthm,esint}
 
\usepackage{graphics,amsmath,amssymb,amsthm,mathrsfs}

\usepackage{graphics,amsmath,amssymb,amsthm,mathrsfs,amsfonts}

\usepackage{sidecap}
\usepackage{float}
\usepackage{extarrows}
\usepackage{booktabs}
\usepackage{verbatim}
\usepackage{hyperref}
\usepackage[usenames,dvipsnames]{xcolor}

\newtheorem{thm}{Theorem}[section]
\newtheorem{lemma}[thm]{Lemma}

\theoremstyle{definition}
\newtheorem{remark}[thm]{Remark}

\def\XXint#1#2#3{{\setbox0=\hbox{$#1{#2#3}{\int}$}
         \vcenter{\hbox{$#2#3$}}\kern-.5\wd0}}

\def\R{\mathbb{R}}

\def\e{\varepsilon}
\def\wN{N^*}

\numberwithin{equation}{section}

\begin{document}

\title{Critical Sets of Elliptic Equations with Rapidly Oscillating Coefficients in  Two Dimensions}

\author{Fanghua Lin \thanks{Supported in part by NSF grant DMS-1955249.}
\qquad
Zhongwei Shen\thanks{Supported in part by NSF grant DMS-1856235 and by Simons Fellowship}}
\date{}
\maketitle
\begin{abstract}

In this paper we continue the study of critical sets of solutions $u_\e$  of second-order elliptic equations in divergence form with rapidly oscillating and periodic coefficients.
In \cite{Lin-Shen-3d}, by controling the "turning" of approximate tangent planes,
we show that the $(d-2)$-dimensional Hausdorff  measures of the critical sets are bounded uniformly with respect to the period $\e$, 
provided that  doubling indices for solutions are bounded.
In this paper we use a different approach, based on the reduction of the doubling indices  of $u_\e$,
to study the two-dimensional case.
The proof relies on the fact that the critical set of a homogeneous harmonic polynomial of degree two or higher in dimension two contains only one point.

\medskip


\medskip

\noindent {\it MR (2010) Subject Classification}: 35J15, 35B27.

\end{abstract}

\section{Introduction}

In this paper we continue the study of critical points of solutions of elliptic equations in homogenization.
More precisely, we consider a family of second-order elliptic operators in divergence form,
\begin{equation}\label{op}
\mathcal{L}_\e =-\text{\rm div} (A(x/\e) \nabla ),
\end{equation}
where $0< \e\le 1$ and  $A(y) = (a_{ij} (y))$ is a  $d\times d$ matrix-valued function in $\R^d$.
Throughout the paper, unless indicated otherwise, we shall assume that 
\begin{itemize}

\item

(ellipticity) there exists some $\lambda\in (0, 1]$ such that
\begin{equation}\label{ellipticity}
\lambda |\xi|^2 \le \langle A(y) \xi, \xi \rangle\quad \text{ and } \quad 
 |\langle A(y)\xi, \zeta\rangle | \le \lambda^{-1} |\xi| |\zeta| \quad \text{ for any } y, \xi, \zeta\in \R^d;
\end{equation}

\item 

(periodicity) $A$ is periodic with respect to some lattice $\Gamma$ of $\R^d$,
\begin{equation}\label{periodicity}
A(y+z)  =A (y) \quad \text{ for any } y\in \R^d \text{ and } z\in \Gamma;
\end{equation}

\item

(smoothness) there exists some $M>0$ such that
\begin{equation}\label{smoothness}
|A(x)-A(y)| \le M |x-y| \quad \text{ for any } x, y \in \R^d.
\end{equation}
\end{itemize}
We will use the notation,
\begin{equation}
E_r = \big\{ x \in \R^d: \ 2\langle (\widehat{A} + (\widehat{A})^T)^{-1} x,  x\rangle <  r^2 \big\}
\end{equation}
for $r>0$,
where $\widehat{A}$ denotes the homogenized matrix for $A$. If $\widehat{A} + (\widehat{A})^T=2I$, then 
$E_r= B(0, r)$.

Let $\chi (y)=(\chi_j (y))$ denote the first-order corrector for $\mathcal{L}_\e$.
We will also assume that the periodic matrix $I +\nabla \chi$ is nonsingular  and that
\begin{equation}\label{inv-0}
\text{det} (I+\nabla \chi)  \ge  \mu
\end{equation}
for some $\mu>0$.
Let $u_\e$ be a non-constant  weak solution of $\mathcal{L}_\e (u_\e)=0$ in  $E_2$ and
\begin{equation}\label{critical-0}
\mathcal{C} (u_\e) =\big\{ x: \ |\nabla u_\e (x)|=0 \big\},
\end{equation}
the critical set of $u_\e$.
Suppose that $u_\e (0)=0$ and
\begin{equation}\label{doubling-0}
\fint_{E_2} u_\e^2 \le 4^N  \fint_{E_1} u_\e^2
\end{equation}
for some $N\ge 1$.
Under  the conditions \eqref{ellipticity}, \eqref{periodicity},  \eqref{smoothness} and \eqref{inv-0},
it is proved in \cite{Lin-Shen-3d} by the present authors that 
\begin{equation}\label{H-M-00}
|  \big\{ x : \text{\rm dist} (x, \mathcal{C}(u_\e)\cap E_{1/2}  )< r \big\}| 
\le C(N) r^2
\end{equation}
for $0< r< 1$, 
and consequently, 
\begin{equation}\label{H-M-0}
\mathcal{H}^{d-2} \big\{ x\in E_{1/2}: \ |\nabla u_\e (x)|=0 \big\}
\le C(N),
\end{equation}
where $C(N)$ depends at most on $d$, $\lambda$, $\Gamma$, $M$, $\mu$,  and $N$.
This is the first result on geometric measure estimates, that are uniform in $\e>0$, for critical sets of solutions of
$\mathcal{L}_\e (u_\e)=0$.
We mention that in  \cite{LS-Nodal},
 the following uniform  bound of the nodal sets,
\begin{equation}\label{Nodal-2}
\mathcal{H}^{d-1} \big\{ x\in E_{1/2}: \ u_\e (x)=0 \big\}
\le C(N),
\end{equation}
was established  by the present authors, under the conditions \eqref{ellipticity}, \eqref{periodicity} and \eqref{smoothness}.
Classical results in the study of nodal, singular, and critical sets for solutions and eigenfunctions of elliptic operators
may be found in \cite{D-Fefferman, Hardt-Simon, Lin-N, Han-1994,  Han-Lin-2, Han-Lin-H-1998, HHN-1999}.
We refer the reader to \cite{Naber-2015, Naber-2017, BET-2017,  Logunov-2018-1, Logunov-2018-2} and their references for more recent work in this area.
Since the bounding constants $C(N)$ depend on the smoothness of coefficients,
the quantitative results for $\mathcal{L}_1$  in the references mentioned above do not extend to the operator $\mathcal{L}_\e$.

The proof of  \eqref{H-M-00} in \cite{Lin-Shen-3d} is based on an estimate of "turning" for the projection of a non-constant solution $u_\e$ onto the subspace 
of spherical harmonic order $\ell$, when the doubling index for $u_\e$ on a sphere $\partial B(0, r)$ is trapped
between $\ell-\delta$ and $\ell +\delta$, for $r$ between $1$ and a minimal radius $r^*\ge C_0 \e$.
In this paper we provide a different and  much simpler proof for the two-dimensional case. 
Our approach is based on the reduction of the doubling index and relies on the fact that the 
critical set of a homogeneous harmonic polynomial of degree $2$ or higher in dimension two contains  only one point.
We note  that
the condition \eqref{inv-0} holds in the case $d=2$ if $A$ is periodic and H\"older continuous.

The following is the main result of the paper.

\begin{thm}\label{main-thm}
Let $d=2$.
Assume that $A=A(y)$ satisfies the conditions \eqref{ellipticity}, \eqref{periodicity} and \eqref{smoothness}.
Let $u_\e\in H^1(E_2)$ be a non-constant weak solution of $\mathcal{L}_\e (u_\e)=0$ in $E_2\subset \R^2$.
Suppose that $u_\e (0)=0$ and \eqref{doubling-0} holds for some $N\ge 1$. Then
\begin{equation}\label{main-1}
\# (E_{1/2}\cap \mathcal{C} (u_\e)) \le C(N),
\end{equation}
where $C(N)$ depends at most on $\lambda$, $\Gamma$, $M$, and $N$.
\end{thm}

Throughout the paper we will use $C$ and $c$ to denote constants that may depend on $d$, 
 $\lambda$ in \eqref{ellipticity}, $\Gamma$ in \eqref{periodicity},  $M$
in \eqref{smoothness}, and $\mu$ in \eqref{inv-0}.
If a constant also depends on other parameters, such as the doubling index of a solution, it will be stated explicitly.


\section{Homogenization}

Let $d\ge 2$ and $A=A(y)$ be a $d\times d$ matrix satisfying  \eqref{ellipticity} and \eqref{periodicity}.
 The  first-order corrector $\chi =\chi(y) =(\chi_j(y) )$ is defined by the cell problem,
\begin{equation}\label{cell-1}
\left\{
\aligned
& \mathcal{L}_1 (\chi_j) =-\frac{\partial}{\partial y_i} \big( a_{ij} \big) \quad \text{ in } Y,\\
& \fint_Y \chi_j=0 \quad \text{ and } \quad
\chi_j \text{ is Y-periodic},
\endaligned
\right.
\end{equation}
for $1\le j\le d$ (the index $i$ is summed from $1$ to $d$),
where $Y$ is the elementary cell  for the lattice $\Gamma$.
The homogenized operator $\mathcal{L}_0$ is given by
\begin{equation} \label{homo-op}
\mathcal{L}_0 = -\text{\rm div} ( \widehat{A} \nabla ),
\end{equation}
where, for $\xi \in \R^d$, 
\begin{equation}\label{homo-co}
\langle \widehat{A}\xi, \xi \rangle = \fint_Y \langle A \nabla v_\xi, \nabla v_\xi \rangle
\end{equation}
and $v_\xi (y)  = \langle \xi, y+ \chi (y) \rangle$.
It follows from  \eqref{cell-1} that $\mathcal{L}_1 (y_j + \chi_j ) =0$ in $\R^d$.
Thus, by De Giorgi - Nash estimates,  $\chi_j$ is H\"older continuous.
Furthermore,   if $A$ is H\"older continuous, i.e.,  there exist $\alpha \in (0, 1]$ and $M_\alpha>0$ such that 
\begin{equation}\label{H}
|A(x)-A(y) |\le M_\alpha |x-y|^\alpha \quad \text{ for any } x, y \in \R^d,
\end{equation}
 so is $\nabla \chi_j$.
 
\begin{thm}\label{2d-inv}
Let $d=2$.
Suppose $A$ satisfies \eqref{ellipticity}, \eqref{periodicity} and \eqref{H}.
Then
\begin{equation}\label{det-0}
\text{\rm det} (I +\nabla \chi) (x) \ge \mu
\end{equation}
for any $x\in \R^2$, where $\mu>0$ depends only on $\lambda$,  $\Gamma$ and $(\alpha, M_\alpha)$.
\end{thm}

\begin{proof}
This theorem was more or less proved in \cite{AN-2001, AN-2018}, although it is not stated explicitly.
Also see related work in \cite{AM-1994, Bauman-2001, AN-2021} and the references therein.
We give an outline of  the proof here.

Step 1. Let $u_i =x_i +\chi_i (x)$ for $i=1, 2$, and $U=(u_1, u_2)$.
Use the continuity and boundedness of $\chi_i$ to show that 
$U: \R^2 \to \R^2$ is onto.

Step 2. Show that $U:  \R^2 \to \R^2$ is one-to-one and $U^{-1}$ is continuous.
As a result,  $U: \R^2 \to \R^2$ is a homeomorphism.  
The smoothness condition \eqref{H} is not needed. See \cite{AN-2001} for details.

Step 3.  Let $\xi\in \R^2$ with $|\xi|=1$. Consider the function $u_\xi = \langle U, \xi \rangle$.
Note that div$(A(x) \nabla u_\xi )=0$ in $\R^2$.
To prove $\text{\rm det} (I+\nabla \chi)>0$, it suffices to show that 
$$
|\nabla u_\xi (x) |>0
$$
for any $x\in \R^2$.
To this end, fix $y_0 \in \R^2$ and $r_0>0$.
Let $\Omega= U^{-1} (B(y_0, r_0))$.
Note that $\text{\rm div} (A\nabla u_\xi)=0$ in $\Omega$, and
$g=u|_{\partial \Omega}$ is unimodal.
This implies $|\nabla u_\xi (x_0)|>0$, where $U(x_0) =y_0$.
See \cite{AN-2018} for details.

Step 4. Use $\text{\rm det} (I+\nabla \chi)>0$ and
a compactness argument to show that $\text{\rm det} (I+\nabla \chi)\ge \mu$, where $\mu>0$ depends only on $\lambda$, $\Gamma$ and $(\alpha, M_\alpha)$.
\end{proof}


\begin{remark} 
  The estimate \eqref{det-0} fails if $d\ge 3$.
See \cite{BMN-2004, C-2015} for a counter-example.
\end{remark}

By a change of variables we may assume that
\begin{equation}\label{h0}
\widehat{A} + (\widehat{A})^T=2I.
\end{equation}
See Remark 2.3 in \cite{Lin-Shen-3d}.
This ensures that solutions of the homogenized equation $\mathcal{L}_0 (u_0)=0$ are harmonic.
The following compactness  theorem will be used in the next section.

\begin{thm}\label{C-thm}
Let $u_j$ be a solution of $\text{\rm div} (A^j (x/\e_j) \nabla u_j)=0$ in $B(0, r_0)$,
where $\e_j \to 0$ and $A^j$ satisfies \eqref{ellipticity}, \eqref{periodicity},  \eqref{H} and \eqref{h0}.
Suppose that $\{ u_j \}$  is bounded in $L^2(B(0, r_0))$.
Then there exists a subsequence, still denoted by $\{u_j\}$ and a harmonic function $u_0$ in $B(0, r_0)$,
such that $u_j \to u_0$ weakly in $L^2(B(0, r_0))$ and weakly in $H^1(B(0, r))$ for any $0< r< r_0$.
Moreover,
\begin{equation}\label{co-3}
\| u_j -u_0\|_{L^\infty(B(0, r))} \to 0,
\end{equation}
\begin{equation}\label{co-4}
\| \nabla u_j - (I +\nabla \chi^j (x/\e_j))\nabla u_0\|_{L^\infty(B(0, r))}
\to 0,
\end{equation}
for any $0< r< r_0$, where $\chi^j$ denotes the first-order correctors for the matrix  $A^j$.
\end{thm}

\begin{proof}
See Theorem 2.7 and Remark 2.8 in  \cite{Lin-Shen-3d}.
\end{proof}


\section{Doubling indices  and critical sets}

Let $d\ge 2$.
As in \cite{Lin-Shen-3d},
we introduce a doubling index for a continuous function $u$ on a ball $B(x_0, r)$, defined by 
\begin{equation}\label{d-0}
\wN (u, x_0, r)
=\log_ 4 \frac{\fint_{\partial B(x_0, r)} (u-u(x_0))^2 }{\fint_{\partial B(x_0, r/2)} (u-u(x_0))^2},
\end{equation}
assuming $\| u-u(x_0)\|_{L^2(\partial B(x_0, t))} \neq 0$ for $0< t\le r$.
Define
\begin{equation} \label{M-2}
\mathcal{M}(\lambda, \Gamma, M)
=\Big\{ A=A(y): \  A \text{ satisfies \eqref{ellipticity}, \eqref{periodicity}, \eqref{smoothness}, and } \eqref{h0}  \Big\}.
\end{equation}

\begin{thm}\label{d1-thm}
Let  $L\ge 2$ and $\delta_0\in (0, 1/2]$.
Assume that $A\in \mathcal{M}(\lambda, \Gamma, M)$.
There exists $\e_0=\e_0 (L, \delta_0)>0$ such that  if $0< \e< \e_0r $ and $u_\e \in H^1(B(x_0, r))$ is a non-constant solution of
$\text{\rm div}(A(x/\e) \nabla u_\e)=0$ in $B(x_0, r)$ for some $r>0$ and $x_0 \in \R^d$, with the properties that,
\begin{equation}\label{d-1}
\wN(u_\e, x_0, r)\le  L+1 \quad \text{ and } \quad
\wN (u_\e, x_0, r/2) \le \ell +\delta_0,
\end{equation}
where $\ell \in \mathbb{N}$ and $1\le \ell \le L$,
 then 
 \begin{equation}\label{d-2}
 \wN(u_\e, x_0, r/4)\le \ell +\delta_0.
 \end{equation}
 If, in addition,  $2^J \e < \e_0 r $ for some integer  $J\ge 0$, then 
 \begin{equation}\label{d-3}
 \wN(u_\e, x_0, r/2^j) \le \ell +\delta_0
\quad \text{  for  } j=2, \dots, J+2.
\end{equation}
\end{thm}

\begin{proof}
This is proved in \cite[Theorem 3.1]{Lin-Shen-3d}.
\end{proof}

\begin{thm}\label{fq-thm}
Let $L\ge 2$ and $\delta_1 \in (0, 1/2]$.
Assume that $A\in \mathcal{M}(\lambda, \Gamma, M)$.
There exists $\e_1=\e_1 (L, \delta_1)>0$ such that if
$0< \e< \e_1r $,
$u_\e\in H^1(B(x_0, r))$,
$\text{\rm div}(A(x/\e)  \nabla u_\e)=0$ in $B(x_0, r)$ for some $x_0\in \R^d$ and $r>0$,
\begin{equation}\label{d-20}
\wN(u_\e, x_0, r)\le L+1   \quad \text{ and } \quad
\wN(u_\e, x_0, r/2) \le \ell -\delta_1,
\end{equation}
where $\ell \in \mathbb{N}$ and $1\le \ell \le L$, 
then
\begin{equation}
\wN (u_\e, x_0, \delta_1 r/(8\ell) ) \le \ell-1 +\delta_1.
\end{equation}
\end{thm}

\begin{proof}
This is proved in \cite[Theorem 3.4]{Lin-Shen-3d}.
\end{proof}

Define
\begin{equation}\label{A}
\mathcal{A}(\lambda, \Gamma, M, \mu)
=\Big\{ A=A(y): A \text{ satisfies \eqref{ellipticity}, \eqref{periodicity}, \eqref{smoothness},  \eqref{inv-0}   and }
\eqref{h0}   \Big\}.
\end{equation}

\begin{thm}\label{low-L}
Let $L\ge 2$ and  $A\in \mathcal{A} (\lambda, \Gamma, M, \mu)$.
There exists $\e_0=\e_0(L)>0$ such that if 
 $u_\e\in H^1(B(0, 1))$ is  a non-constant solution of $\mathcal{L}_\e(u_\e)=0$ in $B(0, 1)$ for some $0< \e< \e_0$, $\wN(u_\e, 0, 1) \le L$, and
\begin{equation}
\wN(u_\e, 0, 1/2)\le 3/2,
\end{equation}
then $|\nabla u_\e (0)|\neq 0$.
\end{thm}

\begin{proof}
This is proved in \cite[Theorem 3.5]{Lin-Shen-3d}.
\end{proof}

Fix $L\ge 1$, $\e>0$, $r>0$, and $x_0\in  \R^d$.  Define
 \begin{equation}\label{F-0}
 \aligned
 \mathcal{F} (L, \e, r, x_0)
 =\Bigg \{
&   u\in H^1(  B(x_0, 2r)):   \ \  u \text{ is not constant, }  u(x_0)=0, \\
 & \text{\rm div} (A(x/\e) \nabla u) =0 \text{ in } B(x_0, 2r) \text{ for some } A \in \mathcal{A}(\lambda, \Gamma, M, \mu),\\ 
& \text{and }  \wN(u, x, r) \le 2L, \  \wN(u, x, r/2) \le L  \   \text{ for all } x\in B(x_0,  r/2) \Bigg\},
 \endaligned
 \end{equation}
 and
 \begin{equation}\label{E-0}
 \aligned
&  \mathcal{E}(L, \e, r) 
 =\\
 & \sup \left\{ \frac{\mathcal{H}^{d-2} (\mathcal{C}(u)\cap B(x_0, r/4))}{r^{d-2}}:\  u \in \mathcal{F} (L, \e, r, x_0)
 \ \ \text{ for some }  x_0\in \R^d \right\},
 \endaligned
 \end{equation}
 where $\mathcal{C}(u)$ denotes the critical set of $u$,
 $$
 \mathcal{C}(u) =\big\{ x: \ |\nabla u(x)|=0 \big\}.
 $$
 Since $u\in \mathcal{F} (L, \e, r, x_0) $ implies $u(\cdot +x_0) \in \mathcal{F}(L, \e, r, 0)$,
 it follows  that 
 \begin{equation}\label{E-00}
   \mathcal{E}(L, \e, r) 
 =
  \sup \left\{ \frac{\mathcal{H}^{d-2} (\mathcal{C} (u)\cap B(0,  r /4))}{r^{d-2}}:\ \   u \in \mathcal{F} (L, \e, r, 0)
 \right\}.
 \end{equation}
 By a simple covering argument, it is not hard to see that if $0< r_1\le  r_2/2$, then
 \begin{equation}\label{E-0a}
 \mathcal{E}(L, \e, r_2) \le C \left(\frac{r_2}{r_1}\right)^2
 \mathcal{E}(L, \e, r_1),
 \end{equation}
 where $C$ depends only on $d$.

  \begin{lemma}\label{E-lemma-1}
 For any $\theta>0$, 
 \begin{equation}\label{E-01}
 \mathcal{E} (L, \e, r) =\mathcal{E} (L, \theta^{-1} \e, \theta^{-1} r).
 \end{equation}
 \end{lemma}
 
 \begin{proof}
 This follows from the observation that if
  $u\in \mathcal{F} (L, \e, r, 0)$ and
 $
 v(x) =u (\theta x),
 $
 then $v\in \mathcal{F} (L, \theta^{-1} \e, \theta^{-1} r, 0)$ and
 $$
 \mathcal{H}^{d-2} ( \mathcal{C} (u) \cap B(0,  r/4))= \theta^{d-2} \mathcal{H}^{d-2} (\mathcal{C} (v) \cap B(0,  \theta^{-1} r/4)).
 $$
 \end{proof}

 \begin{thm}\label{L-lemma}
If $0< r\le \e_0^{-1}  \e$ for some $\e_0>0$, then
\begin{equation}\label{E-03}
\mathcal{E} (L, \e, r) \le C(L, \e_0),
\end{equation}
where $C(L,  \e_0)$ depends on $\e_0$ and $L$.
 \end{thm}
 
 \begin{proof}
 Note that by \eqref{E-01},
 $$
 \mathcal{E}(L, \e, r) =\mathcal{E}(L, r^{-1} \e, 1).
 $$
 Since  $r^{-1} \e \ge \e_0$ and $A$ satisfies \eqref{ellipticity} and \eqref{smoothness}, 
 the estimate \eqref{E-03} follows readily  from \cite{Naber-2017} for the operator $\mathcal{L}_1$
(see \cite{Lin-N, Han-1994} for the case $d=2$ and \cite{Han-Lin-H-1998} for the case of smooth coefficients).
 Indeed, the coefficient matrix $\widetilde{A}(x)=A(x/(r^{-1} \e))$ satisfies the Lipschitz condition $\eqref{smoothness}$
 with  $M\e_0^{-1}$ in the place of $M$. Moreover, 
  the conditions $\wN(u, x, 1)\le 2L$ and $\wN(u, x, 1/2)\le L$ for $x\in B(0, 1/2)$ implies that
 $$
 \int_{B(0, 1)} |\nabla u|^2
 \le C \int_{B(0, 1)} u^2
 \le C \int_{\partial B(0, 1)} u^2,
 $$
 where $C$ depends on $L$.
 The periodicity condition is not needed.
 \end{proof}
  
 \begin{thm}\label{E-2-lemma}
 Fix $L\ge 2$ and $\delta_0\in (0, 1/2]$.
 There exists $\e_0>0$, depending on $L$ and $\delta_0$,  such that if $0< \e< \e_0 r $, $\ell \in \mathbb{N}$ and $2\le \ell \le L$, then
 \begin{equation}\label{E-2-0}
 \mathcal{E}(\ell -\delta_0, \e, r) 
 \le C_0 \mathcal{E} (\ell-1 +\delta_0, \e, c_0 r),
 \end{equation}
 where $c_0 =\delta_0/(8\ell)$ and $C_0$ depends on $L$ and $\delta_0$.
 \end{thm}
 
 \begin{proof}
 In view of  Lemma \ref{E-lemma-1} we may assume $r=1$.
By the definition of $\mathcal{E}(\ell-1+\delta_0, \e, c_0 )$, it suffices to show that 
 if $u\in \mathcal{F}(\ell -\delta_0, \e, 1, 0)$,
 then
 \begin{equation}\label{E-06}
 \wN(u, x_0, c_0) \le \ell-1 +\delta_0 \quad \text{ and } \quad
 \wN (u, x_0, c_0 /2) \le \ell-1 +\delta_0,
  \end{equation}
 for any $x_0\in B(0, 1/2)$,
 provided that $0< \e< \e_0$.
 By covering $B(0, 1/4)$ with a finite number of balls $\{B(y_j, c_0 /4): j=1, 2, \dots, k_0\}$, where $k_0\le  C(d)/ (c_0)^d$ and
  $y_j \in B(0, 1/4)$,
 this would imply  that $u\in \mathcal{F}(\ell-1+\delta_0, \e,  c_0 , y_j )$  for $1\le j \le  k_0$. As a result, 
 $$
 \aligned
 \mathcal{H}^{d-2} (\mathcal{C}(u)\cap B(0, 1/4))
 & \le \sum_j \mathcal{H}^{d-2} (\mathcal{C}(u)\cap B(y_j, c_0 /4))\\
 & \le k_0c_0 ^{d-2}\mathcal{E}(\ell-1+\delta_0, \e, c_0),
 \endaligned
 $$
 from which the estimate \eqref{E-2-0} with $r=1$ follows.
 
 To see \eqref{E-06},  we note that $\wN(u, x_0, 1)\le 2(\ell-\delta_0)$ and $\wN (u, x_0, 1/2) \le \ell -\delta_0$
 for any $x_0\in B(0, 1/2)$.
 By Theorem \ref{fq-thm} we have $\wN(u, x_0, c_0) \le \ell-1 +\delta_0$, where $c_0=\delta_0/(8\ell)$.
 Observe that if $\e$ is sufficiently small, we may use Theorem \ref{d1-thm} to obtain 
 $\wN(u, x_0, 2c_0) \le C(L)$.
 Applying Theorem \ref{d1-thm} again gives $\wN(u, x_0, c_0/2) \le \ell-1 +\delta_0$.
 \end{proof}
 
 \begin{thm}\label{L-lemma-1}
 There exists $\e_0>0$ such that 
 \begin{equation}\label{E-04}
 \mathcal{E} (3/2, \e, r)=0
 \end{equation}
 for any $0< \e< \e_0 r$.
 \end{thm}
 
 \begin{proof}
 By Lemma \ref{E-lemma-1} we may assume $r=1$.
 Let $u\in \mathcal{F}(3/2, \e, 1, 0)$.
 Then $\wN(u, x_0, 1) \le 3$ and $\wN(u, x_0, 1/2)\le 3/2$ for any $x_0\in B(0, 1/2)$.
 By Theorem \ref{low-L} we obtain $|\nabla u(x_0)|\neq 0$, if $0< \e< \e_0$.
 Thus $\mathcal{C}(u) \cap B(0, 1/4)=\emptyset$ and consequently, $\mathcal{E}(3/2, \e, 1)=0$.
 \end{proof}

 
 \section{Proof of Theorem \ref{main-thm}}
 
 Throughout this section we assume $d=2$ and $A$ satisfies \eqref{ellipticity}, \eqref{periodicity},  \eqref{smoothness} and \eqref{h0}.
 Note that by Theorem \ref{2d-inv}, the matrix $A$ satisfies the invertibility  condition \eqref{inv-0}.
 
\begin{lemma}\label{2-d-lemma}
 Let $d=2$ and
 fix $L \ge 2$.
 There exist $\e_0, \delta_0\in (0, 1/4)$, depending  on $L$, such that
 if $0< \e< \e_0r $, $\mathcal{L}_\e (u_\e)=0$ in $B(x_0, r)$, $u_\e$ is not constant,
 \begin{equation}\label{90}
 \wN (u_\e, x_0, r) \le 2(\ell +\delta_0),\ \ \  \wN(u_\e, x_0, r/2) \le \ell +\delta_0,
 \end{equation}
 and 
 $u_\e$ has a critical point in $B(x_0, 3r/4)\setminus B(x_0,   r/128)$, where $\ell \in \mathbb{N}$ and $2\le \ell \le L$,  then
 \begin{equation}\label{91}
 \wN(u_\e, x_0,  r/4 ) \le \ell - \delta_0.
 \end{equation}
 \end{lemma}
 
 \begin{proof}
 By translation and dilation it suffices to consider the case $x_0=0$ and $r=1$.
 To prove  \eqref{91},
we argue by contradiction.
 Suppose there exist  sequences $\{\e_j \}\subset \R_+$ and  $\{ u_j \}\subset H^1(B(0, 1))$  such that
 $\e_j \to 0$, $ \text{\rm div}(A^j (x/\e_j)\nabla u_j )=0$ in $B(0, 1)$ for some
 $A^j$ satisfying \eqref{ellipticity}, \eqref{periodicity},  \eqref{smoothness} and \eqref{h0}, $u_j$ is not constant,
 $\nabla u_j (y_j)=0$ for some $y_j\in B(0, 3/4)\setminus B(0, 1/128)$,
 $$
\wN(u_j, 0, 1) \le 2(\ell + (1/j)) , \ \ \wN(u_j, 0, 1/2) \le \ell + (1/j),
 $$
 and that 
 $$
 \wN(u_j, 0, 1/4) > \ell - (1/j).
 $$
 We may assume that $u_j (0)=0$ and 
 $$
 \fint_{\partial B(0, 1/2)} u_j^2 =1.
 $$
 Since $\wN(u_j, 0, 1) \le 2\ell +2$, this implies that 
  $\{ u_j \}$ is bounded in $L^2(\partial B(0, 1))$.
 It follows that $\{ u_j \}$ is bounded in $L^2(B(0, 1))$. Thus, in view of Theorem \ref{C-thm}, 
 by passing to a subsequence, we may assume that $u_j\to u_0$ weakly in $L^2(B(0, 1))$ and strongly in $L^2(B(0, r))$ for any $0<r<1$,
 where $u_0$ is harmonic in $B(0, 1)$.
 Moreover, 
  \begin{equation}\label{f-100}
 \| u_j -u_0 \|_{L^\infty(B(0, 3/4))} \to 0,
 \end{equation}
 and
 \begin{equation}\label{101}
 \| \nabla u_j -(I+\nabla \chi^j (x/\e_j)) \nabla u_0\|_{L^\infty(B(0, 3/4))} \to 0,
 \end{equation}
where $\chi^j$ denotes the first-order correctors for the matrix $A^j$.
 
 Next, by letting $j \to \infty$, we obtain $u_0(0)=0$ and 
 $$
 1= \fint_{\partial B(0, 1/2)} u_0^2.
 $$
 Hence, $u_0$ is not  constant. Moreover, 
 $$
 \wN(u_0, 0, 1/2) \le \ell \quad \text{ and } \quad 
 \wN (u_0, 0, 1/4) \ge \ell.
 $$
 By the monotonicity of $\wN (u_0, 0, r)$ for harmonic functions, we obtain 
 $$
 \wN(u_0, 0, 1/2)=\wN (u_0, 0,1/4 )=\ell.
 $$
 It follows that $u_0$ is a homogeneous harmonic polynomial of degree $\ell$.
 Since $d=2$, this implies  that $|\nabla u_0 (x)| \neq 0$ for any $x\neq 0$.
 However, since $|\nabla u_j (y_j)| =0$ and 
 $$
  \text{det}(I+\nabla \chi^j (y_j/\e_j) ) \ge \mu>0,
 $$
 in  view of \eqref{101}, we conclude that  $|\nabla u_0 (y_j)| \to 0$ as $j \to \infty$.
 Since $1\ge |y_j | \ge (1/128)$,
 we obtain a contradiction.
    \end{proof}
 
\begin{lemma}\label{2d-lemma-1}
Fix $L\ge 2$.
There exist $ \e_0, \delta_0, \theta \in (0, 1/4)$, depending on $L$,  such that
  if $\e_0^{-1}  \e\le r \le 1$, $\ell\in \mathbb{N}$ and $2\le \ell \le  L$, 
 \begin{equation}\label{f-1}
 \mathcal{E}(\ell +\delta_0, \e, r)
 \le \max \big\{  \mathcal{E} (\ell +\delta_0, \e,   r /2),  C_0 \mathcal{E} (\ell-1 +\delta_0, \e,  \theta  r)
  \big\},
 \end{equation}
 where $C_0$ depends on $L$.
 \end{lemma}
 
 \begin{proof}
 By Lemma \ref{E-lemma-1} we may assume $r=1$.
 Let $u \in \mathcal{F} (\ell +\delta_0, \e, 1, 0)$, where $\delta_0\in (0, 1/4)$ is given by Lemma \ref{2-d-lemma}.
 Consider the cover
 $$
 \{ B(x, 1  /40): x\in \mathcal{C}(u)\cap B(0, 1 /4) \}.
 $$
Let $\{ B(y_j, 1 /40): j=1, 2, \dots, k_0\}$ be a Vitali subcover; i.e., $y_j \in \mathcal{C}(u)\cap B(0, 1/4)$, 
$$
\mathcal{C}(u) \cap B(0, 1/4) \subset \bigcup_{j=1}^{k_0} B(y_j, 1   /8),
$$
and $B(y_i, 1  /40)\cap B(y_j, 1  /40)=\emptyset $ for $i\neq j$. 
We have two cases: $k_0=1$ and $k_0\ge 2$.
Note that $1\le k_0\le C_0$ for some absolute constant $C_0$.

If $k_0=1$, then
$$
\mathcal{C}(u) \cap B(0, 1/4) \subset \mathcal{C}(u) \cap B(y_1,1 /8).
$$
Since  $u\in \mathcal{F}(\ell+\delta_0, \e, 1, 0)$ and $B(y_1, 1 /4) \subset B(0, 1/2)$, we have $u \in \mathcal{F}(\ell+\delta_0,\e,  y_1,1/2 )$.
It follows that 
\begin{equation}\label{f-1-1}
\# (\mathcal{C}(u)\cap B(0, 1/4))
\le \mathcal{E} (\ell+\delta_0, \e,1/2 ).
\end{equation}
 
 Suppose $k_0\ge 2$.
 Then $u \in \mathcal{F}(\ell-\delta_0, \e, y_j, 1/2 )$ for $1\le j \le k_0$.
 Indeed, let $x\in B(y_j, 1 /4)$.
 If $|x-y_j |\ge (1 /128)$, then $y_j \in B(x, 1/2) \setminus B(x, 1 /128)$.
 On the other hand, if $|x-y_j|<( 1 /128)$ and $i\neq j$, 
 then
 $$
 \aligned
 |y_i-x|  & \ge |y_i-y_j| -|y_j -x|\\
&  \ge (1 /20) - (1 /128)\ge (1 /128),
 \endaligned
 $$
 and
 $$
 \aligned
 |y_i-x| & \le |y_i-y_j | + |y_j -x|\\
&  < (1/2) + (1/128)< (3/4).
 \endaligned
 $$
 Hence,  $y_i \in B(x, 3/4) \setminus B(x, 1 /128)$ for $i \neq j$.
 In both cases, 
 by Lemma \ref{2-d-lemma},  we obtain 
 $$
\wN(u, x,1/2 ) \le\ell+\delta_0\le  2(\ell -\delta_0) \quad \text{ and } \quad
\wN(u, x, 1/4 ) \le \ell-\delta_0,
 $$
 for any $x\in B(y_j, 1/4)$, provided that $0< \e< \e_0$.
 As a result, $u \in \mathcal{F}(\ell-\delta_0, \e, y_j, 1/2 )$ for $1 \le j \le k_0$.
 It follows that
 \begin{equation}\label{f-800}
 \aligned
 \# ( \mathcal{C}(u) \cap B(0, 1/4))
&  \le \sum_{j=1}^{k_0}
 \# (\mathcal{C}(u) \cap B(y_j, 1/8 ))\\
 & \le k_0 \mathcal{E} (\ell-\delta_0, \e, 1/2 )\\
 &\le C_0 \mathcal{E} (\ell-1+\delta_0, \e, c_0/2 ),
 \endaligned
 \end{equation}
 where $c_0= \delta_0/(8 \ell)$ and
 we have used Theorem \ref{E-2-lemma} for the last inequality.
 By \eqref{E-0a} we
 may replace $c_0/2$ in \eqref{f-800} by $\theta = \delta_0/(32L)$.
 \end{proof}
  \medskip
  
    We are now in a position to give the proof of Theorem \ref{main-thm}.
  
  \begin{proof}[\bf Proof of Theorem \ref{main-thm}]
  It suffices to consider the case $0< \e< \e_0$, where $\e_0=\e_0(N)>0$ is sufficiently small.
  The case $\e\ge \e_0$ is covered by  \cite{Han-1994, Naber-2017}.
  
  Let $u_\e \in H^1(E_2)$ be a non-constant solution of $ \text{\rm div} (A(x/\e) \nabla u_\e)=0$ in $E_2$, 
  where  $A$ satisfies the conditions \eqref{ellipticity}, \eqref{periodicity} and  \eqref{smoothness}.
  Suppose that $u_\e (0)=0$ and the doubling condition \eqref{doubling-0} holds for some $N\ge 1$.
  Since $d=2$, by Theorem \ref{2d-inv}, the invertibility  condition \eqref{inv-0} is satisfied.
  By a change of variables we may assume $\widehat{A}+( \widehat{A})^T=2I$.
  As a result, $E_r=B(0, r)$ and $u_\e$ satisfies the condition 
  \begin{equation}\label{f-189}
  \fint_{B(0, 2)} u_\e^2 \le  4^N \fint_{B(0, 1)} u_\e^2.
  \end{equation}
  By the doubling inequality for $u_\e$ in \cite[Theorem 1.2]{LS-Nodal},  this gives 
   \begin{equation}\label{f-188}
  \fint_{B(x, r)} u_\e^2 \le  C(N) \fint_{B(x, r/2)} u_\e^2,
  \end{equation}
  for any $x\in B(0, 3/4)$ and  $0< r< 1$. Hence, for $x\in B(0,1/2)$ and $1/2\le r\le 1$,
  $$
  \aligned
  \fint_{\partial B(x, r)} u_\e^2
 &   \le C\fint_{B(x, 5r/4)} u_\e^2\\
&  \le C \fint_{B(x, r/2)} u_\e^2
  \le C \fint_{\partial  B(x, r/2)} u_\e^2,
  \endaligned
  $$
  where $C$ depends on $N$. Consequently, $\wN(u_\e, x, 1) \le L$ and $\wN(u_\e, x, 1/2)\le L$
  for any $x\in B(0, 1/2)$, where $L$ depends on $N$.
  This shows that $u_\e \in \mathcal{F} (L, \e, 1, 0)$ for some integer $L\ge 2$.
  
Let $\e_0, \delta_0 \in (0, 1/4)$ be given by Lemma \ref{2d-lemma-1}.  
We assume that $\e_0$ is so small that Theorem \ref{L-lemma-1} holds.
  We will show that for any $\ell \in \mathbb{N}$ and $1\le \ell \le L$,
  \begin{equation}\label{f-200}
  \mathcal{E} (\ell +\delta_0, \e, r) \le C,
  \end{equation}
  where $0< r\le 1$ and $C$ depends on $L$.
  This yields 
  \begin{equation}\label{f-170}
  \#( \mathcal{C}(u_\e)\cap B(0, 1/4)) \le C(N).
  \end{equation}
  By a simple covering argument we replace $B(0, 1/4)$ in \eqref{f-170} by $B(0, 1/2)$.

To prove  the estimate  \eqref{f-200}, we use an induction argument on $\ell$.
  To this end, we first note that  \eqref{f-200} holds for $\ell =1$. Indeed, if $0< \e< \e_0r $,
  $$
  \mathcal{E} (1+\delta_0, \e, r)
  \le \mathcal{E}(3/2, \e, r)=0,
  $$
  by Theorem \ref{L-lemma-1}.
  If $\e\ge \e_0 r$, we may use Theorem \ref{L-lemma} to obtain 
  $$
  \mathcal{E} (1+\delta_0, \e, r) \le C(L, \e_0).
  $$
   Next, suppose \eqref{f-200} holds for some $\ell<L$.
      If $0< \e< \e_0r$, we use Lemma \ref{2d-lemma-1} to obtain 
   \begin{equation}\label{f-199}
   \mathcal{E} (\ell+1+\delta_0, \e, r) 
   \le \max \big\{ \mathcal{E} (\ell+1+ \delta_0, \e, r/2),   C \big\},
   \end{equation}
   where $C$ depends on $L$.
  By Theorem \ref{L-lemma}, the estimate above also holds for $\e \ge \e_0 r$.
  By an induction argument  on $j$, this implies that 
   \begin{equation}\label{f-198}
   \mathcal{E} (\ell+1+\delta_0, \e, r) 
   \le \max \big\{ \mathcal{E} (\ell+1+ \delta_0, \e, 2^{-j} r),   C \big\}
   \end{equation}
   for any $j \ge 1$. 
   Finally, we  choose $j $ so large that $2^{-j} r \le \e_0^{-1} \e$.
   By Theorem \ref{L-lemma} we obtain 
   $$
   \mathcal{E} (\ell+1+\delta_0, \e, r) \le C,
   $$
   which completes the proof of \eqref{f-200}.

    \end{proof}


 \bibliographystyle{amsplain}
 
\bibliography{Lin-Shen-2021.bbl}

\bigskip

\begin{flushleft}

Fanghua Lin, 
Courant Institute of Mathematical Sciences, 
251 Mercer Street, 
New York, NY 10012, USA.

Email: linf@cims.nyu.edu

\bigskip

Zhongwei Shen,
Department of Mathematics,
University of Kentucky,
Lexington, Kentucky 40506,
USA.

E-mail: zshen2@uky.edu
\end{flushleft}

\bigskip

\medskip

\end{document}